\documentclass[11pt]{article}
\usepackage{amsfonts}

\usepackage[active]{srcltx}
\include{srctex}
\usepackage{amsmath}
\usepackage{graphicx}
\usepackage[
bookmarks=true,         
bookmarksnumbered=true, 
colorlinks=true,
pdfstartview=FitV,
linkcolor=blue, citecolor=blue, urlcolor=blue]{hyperref}

\newtheorem{theorem}{Theorem}[section]

\newtheorem{definition}[theorem]{Definition}

\newtheorem{lemma}[theorem]{Lemma}

\newtheorem{problem}[theorem]{Problem}

\newenvironment{proof}[1][Proof]{\textbf{#1.} }{\ \rule{0.5em}{0.5em}}

\def\be{{\beta}}

\def\be{{\beta}}

\def\ga{{\gamma}}

\def\vare{{\varepsilon}}
\def \eref#1{\hbox{(\ref{#1})}}
\def\th{{\theta}}

\def \eref#1{\hbox{(\ref{#1})}}
\def\th{{\theta}}

\begin{document}
\title{The $\frac 43$-variation of the derivative of the self-intersection Brownian local time and related processes}

\author{Yaozhong Hu\thanks{Y.  Hu is
partially supported by a grant from the Simons Foundation
\#209206.}, \  David Nualart\thanks{ D. Nualart is supported by the
NSF grant DMS0904538. \newline
         }\, and Jian Song   \\
  }
\date{}
\maketitle

\begin{abstract}
In this paper we compute the $\frac 43$-variation of the derivative of the self-intersection  Brownian local time $\gamma_t=\int_0^t \int_0^u \delta '(B_u-B_s)dsdu\,, t\ge 0$, applying techniques
 from  the theory of  fractional martingales
\cite{hns}.
\end{abstract}

\section{Introduction}
Let $B=\{B_t, t\ge 0\} $ be a standard one-dimensional Brownian motion. In this paper we are interested in the process
$\gamma=\{\gamma_t, t\ge 0\}$   formally given  by
\[
\gamma_t=-\frac d{dy} \alpha_t(y)|_{y=0}\,,
\quad {\rm where}\quad
\alpha_t(y)=  \int_0^t \int_0^u \delta_y(B_u-B_s)dsdu  \,.
\]
It can be rigorously defined
as the following limit in $L^2(\Omega)$
\begin{equation} \label{t1}
\gamma_t= \lim_{\epsilon \rightarrow 0} \int_0^t \int_0^u p'_\epsilon(B_u-B_s)dsdu,
\end{equation}
where $p_\epsilon(x)= (2\pi \epsilon) ^{-\frac 12} \exp( -x^2/ (2\epsilon))$.        This process has been studied by Rogers and Walsh in \cite{rw} and by Rosen in \cite{r}.

Let us recall the definition of the $\beta$-variation of a stochastic processes   from \cite{hns}.
\begin{definition}
 Let  $\be \ge 1$ and let  $X=\{X_{t},t\geq 0\}$ be a continuous stochastic process. Denote
 \begin{equation}
S_{\beta ,n}^{[a,b]}(X):=\sum_{i=0}^{n-1}|X_{t_{i+1}^{n}}-X_{t_{i }^{n}}|^{\beta }\,,
\label{e.2.3}
\end{equation}%
 where  $t_{i}^{n}=a+\dfrac{i}{n}(b-a)$ for $i=0,\dots ,n$.  If  the limit of
$ S_{\beta ,n}^{[a,b]}(X)$ exists in  probability  as $n$ tends to infinity, then we say
 that the \ $\beta $-variation of $X$  exists   on the interval $[a,b]$ and the limit is  denoted by $%
\langle X\rangle _{\beta ,[a,b]}$.   We
say that the \ $\beta $-variations  of $X$ on $[a,b]$ exists in $L^{p}$  if
the limit of
$ S_{\beta ,n}^{[a,b]}(X)$  exists in $L^{p}(\Omega)$, where $p\ge 1$.
\end{definition}
For any $a<b<c$, if the $\beta$ variation of $X$ exist on the intervals $[a,b]$ and $[b,c]$, then it also exists on $[a,c]$ and
\[
\langle X \rangle_{\beta,[a,c]}=\langle X \rangle_{\beta,[a,b]}+\langle X \rangle_{\beta,[b,c]}.
\]
Denote by $\{L^x_t, t\ge0, x\in \mathbb{R}\}$ a jointly continuous version of the Brownian local time.
In the paper \cite{rw}  Rogers and Walsh gave  an explicit formula for the exact $\frac 43$-variation of the process $\gamma$, using Gebelein's inequality for Gaussian random variables to bound  the sums of powers of the increments of process
$\gamma$.  More precisely, they  proved the following theorem.
\begin{theorem} \label{th1}
The process  $\gamma$ has a finite $\frac 43$-variation in $L^2$ on any interval $[0,T]$ given by
\[
\langle \gamma\rangle_{\frac 43, [0,T]}= K \int_0^T\left(L^{B_r}_r \right)^\frac{2}{3}dr,
\]
where $K=E|B_1|^\frac{4}{3}E\left[\int_\mathbb{R}(L^z_1)^2dz\right]^\frac{2}{3}$.
\end{theorem}

The purpose of the present paper is to provide an alternative and simpler   proof of Theorem \ref{th1} by using the  methodology introduced by Hu, Nualart and Song in \cite{hns} to compute the $p$-variation of a fractional martingale.  A basic ingredient
in our approach is the stochastic
integral representation of $\gamma_t$ obtained by Hu and Nualart  in  \cite{hn}  through the  Clark-Ocone formula:
\begin{equation}
 \gamma_t= \int_0^t  \left( \int_{\mathbb{R}}
 p_{t-r}(y)  \left( L_r^{y+ B_r  }- L_r^{B_r} \right) dy
 \right)dB_r.\label{gammat}
 \end{equation}
The main idea of the proof is as follows. By an approximation argument, and using the representation of the local time as a semimartingale in the space variable (see Perkins \cite{perkins}), the problem is reduced to the computation of the $\frac 43$-variation of the process
\begin{equation} \label{r1}
X_t= \int_0^t   \left(\int_{\mathbb{R}}  p_{t-r}(y) W_{y} dy \right) dB_r,
\end{equation}
where $W=\{W_y, y\in \mathbb{R}\}$ is a two-sided Brownian motion independent of $B$. Taking into account that $W$ is H\"older continuous of order  almost $\frac 12$, the integral $\int_{\mathbb{R}}  p_{t-r}(y) W_{y} dy$  behaves as $(t-r)^{\frac 14}$ as $r\uparrow t$. In this sense,  the variation of the process $X$ is similar to the variation of the fractional Brownian motion with Hurst parameter $H=\frac 34$. Actually, we can compute easily the $\frac 43$-variation of the process $X$ applying the approach
used for the fractional Brownian motion, based on the decomposition by Mandelbrot and Van Ness \cite{mnv} and the ergodic theorem. Notice, however, that our proof shows only the existence of the $\frac 43$-variation in $L^1$, and we obtain a different expression for the constant $K$ in Theorem \ref{th1}.

The paper is organized as follows. In the next section we derive the $\frac 43$-variation of the process $X$ given in (\ref{r1}) using ergodic theorem. Section 3 is devoted to the proof of Theorem \ref{th1},  where the $\frac 43$-variation is considered in $L^1$. Finally, the appendix contains some technical lemmas.  Along the paper we denote by $C$ a generic constant which may be different from line to line.

\setcounter{equation}{0}
\section{$\frac 43$-variation of a fractional-type process}
Consider the  stochastic process introduced in  (\ref{r1}). This process can also be expressed as
\[
X_t= \int_0^t   E^\theta W_{\theta \sqrt{t-r}}dB_r,
\]
where $\theta$ is a $N(0,1)$ random variable, independent of $B$, and $E^\theta$ denotes the expectation with respect to $\theta$. The following theorem is the main result of this section.
\begin{theorem} \label{th2}
The process $X=\{X_t, t\ge 0\}$ defined in (\ref{r1}) has a finite $\frac 43$-variation in $L^1$ given by
\[
\langle X \rangle_{\frac 43, [a,b]} = K(b-a),
\]
where
\begin{equation} \label{r2}
K= E(|\theta|^{\frac 43}) E
\left| \frac 14 \int_0^\infty \int_0^\infty (x+y) ^{-\frac 32}
(B_{1+x} -B_x) (B_{1+y}- B_y) dxdy \right| ^{\frac 23}.
\end{equation}
\end{theorem}
\begin{proof}
The proof will be done in two steps. To simplify the presentation we assume that $[a,b]=[0,T]$.

\medskip\noindent\textbf{Step  1}
Enlarging the probability space if necessary, we assume that $B=\{B_t, t\in \mathbb{R}\}$ is a two-sided Brownian motion.  Then we define
\[
Y_t=\int_{-\infty}^t E^\theta W_{  \theta\sqrt{t-r}}dB_r-\int_{-\infty}^0 E^\theta W_{  \theta\sqrt{-r}}dB_r.
\]
This process is well defined because, using  the fact that $E(W_xW_y)=\dfrac 12(|x|+|y|-|x-y|)$,  we can write
\begin{eqnarray*}
 E(Y_t^2)&=& E^W \int_\mathbb R \left(E^\theta W_{ \theta\sqrt{(t-r)^+}}-E^\theta W_{ \theta\sqrt{(-r)^+}}\right)^2dr\\
&=&\int_\mathbb R E^{\theta,\eta} E^W\left([ W_{ \theta\sqrt{(t-r)^+}}- W_{ \theta\sqrt{(-r)^+}}][ W_{ \eta\sqrt{(t-r)^+}}- W_{ \eta\sqrt{(-r)^+}}]\right)dr\\
&=&\frac{\sqrt2}{2}E(|\theta|)\int_\mathbb R \left(\sqrt{2[(t-r)^++(-r)^+]}-\sqrt{(t-r)^+}-\sqrt{(-r)^+} \right)dr\\
&=& \frac 1 {\sqrt{\pi}} \left(\int_0^\infty \left(\sqrt{2t+4r}-\sqrt{t+r}-\sqrt{r} \right)dr\right.\\
&&+ \left.\int_0^t \left(\sqrt{2(t-r)}-\sqrt{t-r}\right)dr\right) <\infty.
\end{eqnarray*}
We claim  that the difference
\begin{equation}\label{e.2.2}
Y_t-X_t= \int_{-\infty}^0 \left(E^\theta W_{  \theta\sqrt{t-r}} -E^\theta W_{  \theta\sqrt{-r}}\right)dB_r
\end{equation}
has $\frac 43$-variation in $L^1$  equal to zero
in any time interval $[0,T]$.  In fact, if $t_i=\frac{iT}n$,  then from the Burkholder-Davis-Gundy inequality and the Jensen
inequality,  and using the notation (\ref{e.2.3}),  we have
\begin{eqnarray*}
E S^{[0,T]}_{\frac 43, n} (Y-X)&=& \sum_{i=0}^{n-1}  E\left| \int_{-\infty}^0 \left( E^\theta W_{\theta \sqrt{t_{i+1} -r}}
-E^\theta W_{\theta \sqrt{t_{i} -r}} \right) dB_r \right| ^{\frac 43} \\
& \le & C \sum_{i=0}^{n-1}  E\left( \int_{-\infty}^0 \left( E^\theta W_{\theta \sqrt{t_{i+1} -r}}
-E^\theta W_{\theta \sqrt{t_{i} -r}} \right) ^2 dr \right) ^{\frac 23} \\
& \le & C \sum_{i=0}^{n-1}  \left( \int_{-\infty}^0 E \left( E^\theta W_{\theta \sqrt{t_{i+1} -r}}
-E^\theta W_{\theta \sqrt{t_{i} -r}} \right) ^2 dr \right) ^{\frac 23}.
\end{eqnarray*}
By the same computations as above we obtain
\begin{eqnarray*}
E S^{[0,T]}_{\frac 43, n} (Y-X)&\le& C
 \sum_{i=0}^{n-1}  \left( \int_0^\infty \left( \sqrt{2t_{i+1} +2t_i +4r} -\sqrt{t_{i+1}+r}
 -\sqrt{ t_i+r} \right) dr \right) ^{\frac 23} \\
 &=& C
 \sum_{i=0}^{n-1}  \left( \int_0 ^\infty \int_0^{\frac { t_{i+1} -t_i} 2}\int_0^{\frac { t_{i+1} -t_i} 2}
 (x+y+t_i+r) ^{-\frac 32} dxdydr \right) ^{\frac 23}\\
 &=& C
 \sum_{i=0}^{n-1}  \left( \int_0^{\frac { t_{i+1} -t_i} 2}\int_0^{\frac { t_{i+1} -t_i} 2}
 (x+y+t_i) ^{-\frac 12} dxdy \right) ^{\frac 23}.
\end{eqnarray*}
For $i\ge1$ we use the estimate $ (x+y+t_i) ^{-\frac 12} \le t_i^{-\frac 12}$. In this way we can estimate the above sum   for $i\ge 1$  by
\[
n^{-\frac 43} \sum_{i=1}^{n-1} \left(\frac in\right)^{-\frac 13} =\frac 1n \sum_{i=1}^{n-1} i^{-\frac 13},
\]
which clearly converges to zero as $n$ tends to infinity.

\medskip \noindent \textbf{Step  2}
\
>From Step 1,    it follows that  to prove  Theorem \ref{th2}
it suffices to show
 \begin{equation}
 \langle Y \rangle_{\frac 43, [0,T]} = KT\,.
 \end{equation}
It is easy to verify  that the process $Y$ has stationary increments  and  is self-similar of order $\frac 34$.   As a consequence,  the sequence
  $\{Y_{t_{i+1}} -Y_{t_i}, i\ge 0\}$ has the same law as $ \{\left( \frac Tn \right) ^{\frac 34} \xi_i, i\ge 0\} $, where
\[
\xi_i=  \int_{-\infty}^{i+1} E^\theta W_{\theta\sqrt{i+1-r}}dB_r-\int_{-\infty}^{i} E^\theta W_{\theta\sqrt{i-r}}dB_r.
\]
It suffices to show that $\frac 1n \sum_{i=0}^{n-1} |\xi_i |^{\frac 43}$ converges in $L^1$ to $K$.   By the ergodic theory,  we know that,
\[
 \lim_{n\to\infty} \frac 1n\sum_{i=1}^{n} |\xi_i|^\frac43 = Z=E(|\xi_1 |^{\frac 43} | \mathcal{I}),
 \]
 in $L^1$, where $\mathcal I$ is the invariant $\sigma$-field. We claim that the random variable $Z$ is a constant. To prove this we will show that both  random variables $E^WZ$ and $E^BZ$ are constant, where $E^W$ and $E^B$ denote, respectively, the mathematical expectation with respect to the processes $W$ and $B$.

Let us first compute $E^WZ$. Let  $C_0=E|\theta|^\frac 43$. Then, we can write
\begin{eqnarray*}
E^W |\xi_i|^\frac43
&=& C_0 \left(E^W\left(\int_{-\infty}^{i+1} E^\theta W_{\theta\sqrt{i+1-r}}dB_r-\int_{-\infty}^{i} E^\theta W_{\theta\sqrt{i-r}}dB_r\right)^2\right)^\frac23\\
&=& C_0\Bigg(\int_{-\infty}^{i+1}\int_{-\infty}^{i+1} E^W\Big[\left(E^\theta W_{\theta\sqrt{i+1-r}}-E^\theta W_{\theta\sqrt{(i-r)^+}}\right)\\
&& \qquad \times \left(E^\eta W_{\eta\sqrt{i+1-s}}-E^\eta W_{\eta\sqrt{(i-s)^+}}\right)\Big]dB_rdB_s \Bigg)^\frac23\,,
\end{eqnarray*}
where the double integral  $\int\int \cdots   dB_rdB_s$  with
respect to $B$ is a Stratonovich-type integral. Thus,
\begin{eqnarray*}
E^W |\xi_i|^\frac43
&=&C_0\left(\dfrac 12 \int_{-\infty}^{i+1}\int_{-\infty}^{i+1} \left(-\sqrt{(i+1-s)+(i+1-r)}-\sqrt{(i-r)^++(i-s)^+}\right.\right. \\
&&\quad\left.  \left.+\sqrt{(i+1-r)+(i-s)^+}+\sqrt{(i+1-s)+(i-r)^+}\right) dB_rdB_s \right)^\frac23\\
&=&C_0\left(\dfrac 14 \int_{-\infty}^{i+1}\int_{-\infty}^{i+1} \int_{(i-r)^+}^{i+1-r}\int_{(i-s)^+}^{i+1-s}(x+y)^{-\frac32}dydxdB_rdB_s \right)^\frac23
\,.
\end{eqnarray*}
One can  exchange the  integration order of $x,y$ and $r, s$. The domain
$-\infty<r,s,<i+1\,, (i-r)^+<x<i+1-r\,, (i-s)^+<y<i+1-s$ can be written as
$0<x,y<\infty\,, i-x<r<i+1-x\,, i-y<s<i+1-y$.  Thus, we have
\begin{eqnarray*}
E^W |\xi_i|^\frac43&=&C_0\left(\dfrac14 \int_0^\infty\int_0^\infty(x+y)^{-\frac32}(B_{i+1-x}-B_{i-x})(B_{i+1-y}-B_{i-y})dydx\right)^\frac23\\
&=&C_0\left(\dfrac14 \int_0^\infty\int_0^\infty\dfrac 1{\Gamma(\frac32)}\int_0^\infty e^{-(x+y)z}z^\frac12dz(B_{i+1-x}-B_{i-x})(B_{i+1-y}-B_{i-y})dydx\right)^\frac23\\
&=&C_0\left(\dfrac1{4\Gamma(\frac32)}\int_0^\infty\left(\int_0^\infty (B_{i+1-x}-B_{i-x})e^{-xz}dx\right)^2z^\frac12dz \right)^\frac23.
\end{eqnarray*}
For any fixed $x$ and $y$  in $\mathbb{R}$,    the correlation between the Gaussian random variables $B_{1-x} -B_{-x}$ and $B_{i+1-y} -B_{i-y}$ is   zero when  $i$  is sufficiently large.   This implies that the sequence
\[
\int_0^\infty\left(\int_0^\infty (B_{i+1-x}-B_{i-x})e^{-xz}dx\right)^2z^\frac12dz
\]
is stationary and ergodic.  As a consequence, $\frac 1n
\sum_{i=0}^{n-1} E^W |\xi_i|^{\frac 43}$ converges to the constant
$K$ given  in (\ref{r2}).

Finally, we show that $E^BZ$ is constant.   We can write
\begin{eqnarray*}
E^B |\xi_i|^\frac43&=& C_0 \left(\int_{\mathbb{R}}\left(E^\theta W_{\theta\sqrt{(i+1-r)^+}} -E^\theta W_{\theta\sqrt{(i-r)^+}} \right)^2dr\right)^\frac23.
\end{eqnarray*}
For any fixed $r$ and $s$ in $\mathbb{R}$, the covariance between the random variables $\eta_0(s)$ and $\eta_i(r)$, where
\[
\eta_i(r) =E^\theta W_{\theta\sqrt{(i+1-r)^+}} -E^\theta W_{\theta\sqrt{(i-r)^+}},
\]
is given by
\begin{eqnarray*}
E^W(\eta_0(s) \eta_i(r) )&=&
\frac 12 E(|\theta|)  \Big(  -\sqrt{(i+1-r)^+ + (1-s)^+} +\sqrt{(i+1-r)^+ + (-s)^+}\\
&&+\sqrt{(i-r)^+ + (1-s)^+}-\sqrt{(i-r)^+ + (-s)^+} \Big),
\end{eqnarray*}
and it converges to zero as $i$ tends to infinity. Again, this implies that the sequence
\[
\int_{\mathbb{R}}\left(E^\theta W_{\theta\sqrt{(i+1-r)^+}} -E^\theta W_{\theta\sqrt{(i-r)^+}} \right)^2dr
\]
is stationary and ergodic, and as a consequence,  $\frac 1n \sum_{i=0}^{n-1}
E^B |\xi_i|^{\frac 43}$ converges to a constant.
\end{proof}

\setcounter{equation}{0}
 \section{Proof of Theorem \ref{th1}}

 In this section we proceed to the proof of Theorem \ref{th1}, where the $\frac 43$-variation is in $L^1(\Omega)$, and the constant $K$ has the alternative expression given by (\ref{r2}).

 Fix a partition $s_k= \frac {kT}N$, $k=0,\dots, N$. For any point $t$ we denote by $t(N)$
 the maximum point of the partition
  on  the left of $t$, namely, $t(N)=t_k$ if $s_k\le t<s_{k+1}$. We  approximate the process $\ga_t$
  defined in \eref{gammat} by a sequence of processes
   obtained by freezing the time coordinate of $\displaystyle L_r^{y+ B_r  }- L_r^{B_r} $ at the point $r=r(N)$, that is,
 \[
 \gamma^N_t= \int_0^t   \int_{\mathbb{R}}     p_{t-r}(y)  \left( L_{r(N)}^{y+ B_{r}  }- L_{r(N)}^{B_{r}} \right) dy dB_r.
 \]
 The proof will be divided into several steps.

 \medskip \noindent
 \textbf{Step  1}\quad
 We claim that
 \begin{equation} \label{e1}
 \lim_{N\rightarrow \infty}  \limsup_{n\rightarrow \infty} \sum_{k=0}^{N-1}
 ES^{[\frac {kT}N, \frac{(k+1)T} N]}_{\frac 43,n} (\gamma -  \gamma^{N}) =0.
\end{equation}
 Consider a uniform partition of the interval $[kT/N, (k+1)T/N]$ denoted by
 $r_0<r_1< \cdots  < r_n$, where $r_j= \frac {kT}N + \frac {jT}{nN}$,  $j=0, 1, \dots, n$.  Then,
  \begin{equation}
S_{n,N}^k:=S^{[\frac {kT}N, \frac{(k+1)T} N]}_{\frac 43,n} (\gamma -
\gamma^{N})= \sum_{j=0}^{n-1}|\Delta _{j} (\gamma -\gamma^N)|^{\frac
43 }\,, \label{e.3.2}
 \end{equation}
 where $\Delta _{j} (\gamma- \gamma^N) = (\gamma -  \gamma^{N})_{r_{j+1} }- (\gamma -  \gamma^{N})_{r_{j } }$.
 Let $f^N_r(y)= L_r^{y+ B_r  }- L_r^{B_r}-  L_{r(N) }^{y+ B_{r } }+ L_{r(N) }^{B_{r }}$. Then,
\[
(\gamma  -\gamma^N)_t=\int_0^t  \int_{\mathbb{R}} p_{ t-r}(y) f^N_r dydB_r.
\]
As a consequence,
\begin{eqnarray}
 S_{n,N}^k &=& \sum_{j=0}^{n-1}\left|\int_0  ^{r_{j+1}}     \int_{\mathbb{R}}
 p_{ r_{j+1}-r}(y) f^N_r dy  dB_r -\int_0 ^{r_{j }}      \int_{\mathbb{R}}
 p_{ r_{j}-r}(y) f^N_r dy dB_r\right|^\frac43\nonumber\\
&= & \sum_{j=0}^{n-1}\left|\int_{r_j}  ^{r_{j+1}}      \int_{\mathbb{R}}
p_{ r_{j+1}-r}(y) f^N_r  dy  dB_r +\int_0 ^{r_{j }}      \int_{\mathbb{R}} [p_{ r_{j+1}-r}
(y)-p_{r_j-r}(y)] f^N_r dy dB_r\right|^\frac43\nonumber\\
&\le & C\sum_{j=0}^{n-1}\left( \left|\int_{r_j}  ^{r_{j+1}}
\int_{\mathbb{R}}     p_{ r_{j+1}-r}(y) f^N_r  dy  dB_r
\right|^\frac43+ \left|\int_0 ^{r_{j }}
   \int_{\mathbb{R}} [p_{ r_{j+1}-r}(y)-p_{r_j-r}(y)] f^N_r dy dB_r\right|^\frac43\right)\nonumber\\
&=&C \sum_{j=0}^{n-1}(|\Gamma_j^k|^\frac
43+|\Phi_j^k|^\frac43),\label{e.3.3}
\end{eqnarray}
 where
  \begin{eqnarray*}
 \Gamma_j^k&=& \int_{r_j}  ^{r_{j+1}}    \int_{\mathbb{R}}     p_{ r_{j+1}-r}(y)  \left( L_r^{y+ B_r  }- L_r^{B_r}-  L_{r(N) }^{y+ B_{ r } }+ L_{r(N) }^{B_{r }} \right) dy dB_r \\
 &=& \int_{r_j}  ^{r_{j+1}}        E(  L_r^{ B_{r_{j+1}}  }- L_r^{B_r}-  L_{r(N) }^{ B_{r_{j+1}}   }+ L_{r(N) }^{B_{r }}  |\mathcal{F}_r)   dB_r,
 \end{eqnarray*}
 and
   \begin{eqnarray*}
 \Phi_j^k&=& \int_0 ^{r_{j}}     \int_{\mathbb{R}}    [ p_{ r_{j+1}-r}(y) -  p_{ r_{j }-r}(y) ]\left( L_r^{y+ B_r  }- L_r^{B_r}-  L_{r (N) }^{y+ B_{r } }+ L_{r_(N) }^{B_{r }} \right) dydB_r \\
 &=& \int_0 ^{r_j}    E(  L_r^{ B_{r_{j+1}}  }- L_r^{B_{r_{j }}}-  L_{r(N) }^{ B_{r_{j+1}}   }+ L_{r(N) }^{B_{r_{j }}  }  |\mathcal{F}_r)   dB_r.
 \end{eqnarray*}
 Therefore,
 \[
 ES_{n,N}^k \le C\left(        \sum_{j=0}^{n-1}  E( | \Gamma^k_j|^{\frac 43 } ) +
  \sum_{j=0}^{n-1} E(  | \Phi^k_j|^{\frac 43 } )\right).
 \]
Using the Burkholder  inequality we obtain
\[
 E( | \Gamma^k_j|^{\frac 43 } )  \le C  E \left(
  \int_{r_j}  ^{r_{j+1}}        E(  L_r^{ B_{r_{j+1}}  }- L_r^{B_r}-
   L_{r(N) }^{ B_{r_{j+1}}  }+ L_{r(N) }^{B_{r  }}  |\mathcal{F}_r)    ^2 dr \right) ^{\frac 23},
  \]
  and
 \[
 E(  | \Phi^k_j|^{\frac 43 }) \le C E
\left(  \int_0 ^{r_{j}}       E(  L_r^{ B_{r_{j+1}}  }- L_r^{B_{r_{j }}}-
  L_{r(N)}^{ B_{r_{j+1}}   }+ L_{r(N) }^{B_{r_{j }}  }  |\mathcal{F}_r)    ^2 dr \right)^{ \frac 23}.
\]
Let us first prove that
\begin{equation} \label{e2}
 \lim_{N\rightarrow \infty}  \limsup_{n\rightarrow \infty} \sum_{k=0}^{N-1} \sum_{j=0}^{n-1}
 E( | \Gamma^k_j|^{\frac 43 } )=0.
\end{equation}
We shall use the notation $L_{a, b}^x=L_b^x-L_a^x$.  Then,    %
 we can write
 \begin{equation} \label{e3}
E( | \Gamma^k_j|^{\frac 43 } ) \le C  \left( E \int_{r_j}  ^{r_{j+1}}
 (  L_{[r(N),r]}^{B_{r_{j+1}}} - L_{[r(N),r]}^{B_r})^2 dr \right)^{\frac 23}.
\end{equation}
Consider the Brownian motion $B_t-B_u$  where the parameter $u$ goes backward from $t$ to $0$.
    Then, Tanaka's formula applied to this Brownian motion says that for any $s<t$
 \[
 (B_t-B_s -x)_+ -(-x)_+ =  -\int_s^t \mathbf{1}_{\{B_t-  B_u>x \}} \widetilde d B_u +
 \frac 12 \int_s^t \delta_x(B_t-B_u)du,
 \]
 where $ \widetilde d$ denotes the backward It\^o integral. Making the change
 of variable $x=B_t-B_\tau $, $\tau>t$  yields
  \begin{equation}
 (B_\tau -B_s  )_+ -(B_\tau-B_t)_+ =
  -\int_s^t \mathbf{1}_{\{B_u <B_\tau \}} \widetilde d B_u + \frac 12 \int_s^t \delta_{B_\tau} (
  B_u)du\,.
  \label{tanaka}
 \end{equation}
 Therefore, letting $s=r(N)$, $t=r$ and $\tau={r_{j+1}}$ in the above equality yields
 \[
 (B_{r_{j+1}} - B_{r(N)} )_+    -  (B_{r_{j+1}} - B_r )_+   =
 -\int_{r(N)}^r  \mathbf{1}_{\{ B_u< B_{r_{j+1}} \}}\widetilde d B_u + \frac 12  L_{[r(N),r]}  ^{ B_{r_{j+1}}}.
 \]
 On the other hand, letting  $s=r(N)$ and $t=\tau =r$  gives us
  \[
 (B_r - B_{r(N)} )_+       =
 -\int_{r(N)}^r  \mathbf{1}_{\{ B_u< B_r \}}\widetilde d B_u + \frac 12  L_{[r(N),r]}  ^{ B_ r }.
 \]
This implies that
\begin{eqnarray*}
\left|L_{[r(N),r]}^{B_{r_{j+1}}} - L_{[r(N),r]}^{B_r}\right|
&\le &2\left|   ( B_{r_{j+1}} - B_{r(N)} )_+     -  (B_r - B_{r(N)} )_+ \right|  +2 (B_{r_{j+1}} - B_r )_+    \\
  && \qquad +2 \left|\int_{r(N)}^r   \left( \mathbf{1}_{\{ B_u< B_{r_{j+1}} \}}   - \mathbf{1}_{\{ B_u< B_r \}}   \right)\widetilde d B_u\right|\\
  &\le &4\left|    B_{r_{j+1}} -     B_r       \right|  +2 \left|\int_{r(N)}^r   \left( \mathbf{1}_{\{ B_u< B_{r_{j+1}} \}}   - \mathbf{1}_{\{ B_u< B_r \}}   \right)\widetilde d B_u\right|\,.
\end{eqnarray*}
Therefore,
\begin{equation} \label{e4}
E\left(   L_{[r(N),r]}^{B_{r_{j+1}}} - L_{[r(N),r]}^{B_r} \right)^2 \le
32(r_{j+1}-r) + 8
 \int_{r(N)}^r  E \left( \mathbf{1}_{\{ B_u< B_{r_{j+1}} \}}   - \mathbf{1}_{\{ B_u< B_r \}}   \right)^2 du.
 \end{equation}
 Notice that
 \begin{eqnarray*}
   E \left( \mathbf{1}_{\{ B_u< B_{r_{j+1}} \}}   - \mathbf{1}_{\{ B_u< B_r \}}   \right)^2  &=&      P( B_r  < B_u< B_{r_{j+1}} )      +      P( B_r > B_u> B_{r_{j+1}} )  \,.
 \end{eqnarray*}
 Using the density of two-dimensional Gaussian random variables one can see that the probability
 $P( B_r  \le B_u< B_{r_{j+1}} ) $ is bounded by a constant times $
 \frac {\sqrt{r_{j+1} -r}}   {\sqrt{r-u}}$, which implies
 \begin{equation} \label{e5}
 \int_{r(N)}^r  E \left( \mathbf{1}_{\{ B_u< B_{r_{j+1}} \}}   - \mathbf{1}_{\{ B_u< B_r \}}   \right)^2 du
 \le C \sqrt{ r_{j+1} -r_j} N^{-\frac 12}.
 \end{equation}
 From (\ref{e3}), (\ref{e4}) and (\ref{e5}) we obtain
 \begin{eqnarray*}
 E( | \Gamma^k_j|^{\frac 43 } )&\le & C\left( (r_{j+1} -r_j)^2 +( r_{j+1} -r_j)^{\frac 32} N^{-\frac 12} \right) ^{\frac 23}\\
&\le & C \left( n^{-2} N^{-2} + n^{ -\frac 32} N^{-2} \right)^{\frac 23} \\
&\le & C \left( n^{-\frac 43} N^{-\frac 43} + n^{ -1} N^{-\frac 43} \right),
 \end{eqnarray*}
 which implies (\ref{e2}).

To complete the proof of (\ref{e1}),  we need  to show that
\begin{equation} \label{e5.5}
 \lim_{N\rightarrow \infty}  \limsup_{n\rightarrow \infty} \sum_{k=0}^{N-1} \sum_{j=0}^{n-1}
 E( | \Phi^k_j|^{\frac 43 } )=0.
\end{equation}
We continue to use  the same notation as above. It is easy to obtain
by using the Burkholder inequality
   \[
E( | \Phi^k_j|^{\frac 43 } ) \le   \left( E \int_0 ^{r_{j}}    ( E( L_{[r(N),r]}^{B_{r_{j+1}}} - L_{[r(N),r]}^{B_{r_j}} | \mathcal{F}_r)^2 dr \right)^{\frac 23}.
   \]
In order to deal with the above  term,  we use the backward Tanaka
formula \eref{tanaka} again by taking $\tau=r_{j+1}$ and $r_j$.
Subtracting the two obtained equations, we obtain
\begin{equation} \label{e7}
  L_{[r(N),r]}^{B_{r_{j+1}}} - L_{[r(N),r]}^{B_{r_j}}= C_j(r)+D_j(r),
  \end{equation}
  where
  \[
  C_j(r)=2\left( (B_{r_{j+1}} - B_{r(N)} )_+    -  (B_{r_{j+1}} - B_r )_+  -  (B_{r_j} - B_{r(N)} )_+
 + (B_{r_{j}} - B_r )_+ \right),
 \]
 and
 \[
 D_j(r)= 2 \int_{r(N)}^r   \left( \mathbf{1}_{\{ B_u< B_{r_{j+1}} \}}   - \mathbf{1}_{\{ B_u< B_{r_j} \}}   \right)\widetilde d B_u.
\]
   Notice that
\begin{eqnarray*}
&& E[   (B_{r_{j+1}} - B_{r(N)} )_+- (B_{r_j} - B_{r(N)} )_+ | \mathcal{F}_r]\\
&&\qquad  = E^\xi[   ( \sqrt{r_{j+1}-r} \xi + B_r - B_{r(N)} )_+- (\sqrt{r_{j  }-r} \xi + B_r - B_{r(N)} )_+] ,
\end{eqnarray*}
 where $\xi$ is $N(0,1)$. Hence,
 \[
| E[   (B_{r_{j+1}} - B_{r(N)} )_+- (B_{r_j} - B_{r(N)} )_+ | \mathcal{F}_r]|
\le C( \sqrt{r_{j+1}-r} -\sqrt{r_{j }-r}).
\]
Therefore, we obtain
\begin{eqnarray*}
   \int _0^{r_j} E(C_j(r) | \mathcal{F}_r)^2 dr
  &\le & C\int_0  ^{r_{j}}  ( \sqrt{r_{j+1}-r} -\sqrt{r_{j }-r})^2 dr   \\
&\le &   C \int_0^{r_j} (r_{j+1}- r_j)^{\frac 74 } (r_j-r)^{-\frac 34} dr \le  C (nN)^{-\frac 74}.
\end{eqnarray*}
As a consequence,
\begin{equation} \label{e6}
\lim_{N\rightarrow \infty}  \limsup_{n\rightarrow \infty}
\sum_{k=0}^{N-1} \sum_{j=0}^{n-1}  \left(E   \int _0^{r_j} E(C_j(r) | \mathcal{F}_r)^2 dr
\right)^{\frac 23}  =0.
\end{equation}
 For the second term  in the decomposition (\ref{e7}) we can write
\begin{eqnarray}  \label{e8}
&&E \int _0^{r_j} E(D_j(r) | \mathcal{F}_r)^2 dr \\
&&\qquad \le \notag
 \int_0  ^{r_{j}}   \int_{r(N)}^r   E \left[ E\left( \mathbf{1}_{\{ B_{r_j} <B_u < B_{r_{j+1}} \}}   - \mathbf{1}_{\{ B_{r_j} >B_u> B_{r_{j+1}} \}}  | \mathcal{F}_r  \right)  \right]^2  du dr.
\end{eqnarray}
>From Lemma \ref{A1} it follows that
\begin{eqnarray*}
&&E \left[ E\left( \mathbf{1}_{\{ B_{r_j} <B_u < B_{r_{j+1}} \}}   - \mathbf{1}_{\{ B_{r_j} >B_u> B_{r_{j+1}} \}}  | \mathcal{F}_r  \right)  \right]^2  \\
&&\qquad \le C (r-u)^{-\frac 12}  \left( 2\sqrt{2(r_j-r) +\frac T{nN}}
-\sqrt{2(r_j-r)} -\sqrt{ 2(r_j-r) +\frac {2T}{nN}} \right).
\end{eqnarray*}
Substituting this expression into (\ref{e8}) yields
 \begin{eqnarray*}
 &&E \int _0^{r_j} E(D_j(r) | \mathcal{F}_r)^2 dr \\
 &\le &  C  \int_{0}^{r_j}   \int_{r(N)} ^r
 (r-u)^{-\frac 12} \\
 &&\times \left( 2\sqrt{ 2(r_j-r)+ \frac T{nN}} - \sqrt{2(r_j-r)} -\sqrt{2(r_j-r)+2\frac T{nN}} \right)dudr\\
 &\le& C N^{-\frac 12} \int_{0}^{r_j}
  \left( 2\sqrt{ 2(r_j-r)+ \frac T{nN}} - \sqrt{2(r_j-r)} -\sqrt{2(r_j-r)+2\frac T{nN}} \right) dr \\
 &\le& CN^{-\frac 12}  \Bigg(  2\left( 2\left(\frac kN+\frac j{Nn}\right)+\frac 1{Nn}\right)^\frac32 - 2\left( \frac 1{Nn}\right)^\frac32\\
  &&-\left( 2\left(\frac kN+\frac j{Nn}\right)\right)^\frac32 -\left( 2\left(\frac kN+\frac j{Nn}\right)+ 2\frac 1{Nn}\right)^\frac32+\left(  2\frac 1{Nn}\right)^\frac32  \Bigg) \\
     &\le&   C N^{-2} n^{-\frac 32}     \sup_{j,n}
      \left( 2 ( 2(nk+j)+1) ^{\frac 32}-2  - (2(nk+j)) ^{\frac 32} -(2(nk+j)+2) ^{\frac 32}+(2 ) ^{\frac 32}\right)^{\frac 23}\\
       &=&   C  N^{-2} n^{-\frac 32} \sup_j
      \left( 2 ( 2j+1) ^{\frac 32}-2  - (2j) ^{\frac 32} -(2j+2) ^{\frac 32}+(2 ) ^{\frac 32}\right)^{\frac 23}\\
      &\le& CN^{-2} n^{-\frac 32}.
    \end{eqnarray*}
    Therefore,
  \[
\sum_{k=0}^{N-1} \sum_{j=0}^{n-1}  \left(E   \int _0^{r_j} E(D_j(r) | \mathcal{F}_r)^2 dr
\right)^{\frac 23}  \le C N^{-\frac 13},
\]
which implies
\begin{equation}\label{e9}
\lim_{N\rightarrow \infty}  \limsup_{n\rightarrow \infty}\sum_{k=0}^{N-1} \sum_{j=0}^{n-1}  \left(E   \int _0^{r_j} E(D_j(r) | \mathcal{F}_r)^2 dr
\right)^{\frac 23}=0.
\end{equation}
Then, (\ref{e6}) and (\ref{e9}) imply (\ref{e5.5}), which completes the proof of  (\ref{e1}).

\medskip \noindent
\textbf{Step  2}
Define
\[
\gamma^{N,1}_t=\int_{t(N)}^t     \int_{\mathbb{R}}p_{t-r}(y) \left( L_{r(N)}^{y+ B_{r}  }- L_{r(N)}^{B_{r}} \right) dy dB_r.
\]
We claim that, for each fixed $N$,
\[
 \left \langle   \gamma^N - \gamma^{N,1} \right\rangle_{\frac 43, [0,T]}=0.
\]
It suffices to show that for each $k=0, \dots, N-1$,  the $\dfrac 43$-variation
of $\gamma^N- \gamma^{N,1}$ over the interval $[kT/N, (k+1)T/N)$ is zero.  When $t\in [kT/N, (k+1)T/N), t(N)=kT/N$, and
\[
(\gamma^N-\gamma^{N,1})(t)=\int_0^{\frac kNT}\int_\mathbb R p_{t-r}(y)(L_{r(N)}^{y+B_r}-L_{r(N)}^{B_r})dydB_r.
\]
With the same notation as in Step 1, set
\[
S_{n,N}:=S^{[\frac {kT}N, \frac{(k+1)T} N]}_{\frac 43,n} (\gamma^N -  \gamma^{N,1})= \sum_{j=0}^{n-1}|\Delta _{j} (\gamma^N -\gamma^{N,1})|^{\frac 43 },
 \]
 where
 \begin{eqnarray*}
 \Delta _{j} (\gamma^N -\gamma^{N,1})
 &=& \int_0^{\frac {kT}N}\int_\mathbb R (p_{r_{j+1} -r} (y)- p_{r_{j } -r} (y))\left( L_{r(N)}^{y+ B_r  }- L_{r(N)}^{B_r} \right)dydB_r \\
 &=& \int_0^{\frac {kT}N}  \int_0^{r(N)} (p_{r_{j+1} -r}(B_r-B_s) -p_{r_{j} -r}(B_r-B_s)dsdB_r.
 \end{eqnarray*}
Applying the Burkholder  inequality yields
\begin{eqnarray*}
&&E| \Delta _{j} (\gamma^N -\gamma^{N,1})|^{\frac 43}\\
&&\qquad  \le
C E\left(\int_0^{\frac {kT}N} \left( \int_0^{r(N)} (p_{r_{j+1} -r}(B_r-B_s) -p_{r_{j} -r}(B_r-B_s))ds\right)^2dr  \right)^{\frac 23} \\
&&\qquad  \le  C \left(\int_0^{\frac {kT}N} E \left( \int_0^{r(N)} (p_{r_{j+1} -r}(B_r-B_s) -p_{r_{j} -r}(B_r-B_s))ds\right)^2dr  \right)^{\frac 23}.
\end{eqnarray*}
Then, for any $u<s< r(N)<r \le t(N) \le r_j<r_{j+1}$  we can write, using Lemma \ref{A2}
\begin{eqnarray*}
&&E\left( (p_{r_{j+1} -r}(B_r-B_s) -p_{r_{j} -r}(B_r-B_s)) (p_{r_{j+1} -r}(B_r-B_u) -p_{r_{j} -r}(B_r-B_u)) \right)  \\
&&\quad  = \left( (r_{j+1}-s) (r_{j+1}-r+s-u) + (r_{j+1}-r)(r-s) \right) ^{-\frac 12} \\
&&\qquad -  \left( (r_{j+1}-s) (r_{j}-r+s-u) + (r_{j+1}-r)(r-s) \right) ^{-\frac 12} \\
&&\qquad -  \left( (r_{j}-s) (r_{j+1}-r+s-u) + (r_{j}-r)(r-s) \right) ^{-\frac 12} \\
&&\qquad +  \left( (r_{j}-s) (r_{j}-r+s-u) + (r_{j}-r)(r-s) \right) ^{-\frac 12}\\
&&\quad =-\frac 12 \int_{r_j} ^{r_{j+1}}
\left( (r_{j+1} -s) (\theta -r+s-u) + (r_{j+1}-r) (r-s) \right) ^{-\frac 32}(r_{j+1}-s) d\theta\\
&& \qquad +  \frac 12 \int_{r_j} ^{r_{j+1}}
\left( (r_{j} -s) (\theta -r+s-u) + (r_{j}-r) (r-s) \right) ^{-\frac 32}(r_{j}-s) d\theta.
\end{eqnarray*}

Integrating in the variable $u$ yields
\begin{eqnarray*}
&&\int_0 ^sE\left( (p_{r_{j+1} -r}(B_r-B_s) -p_{r_{j} -r}(B_r-B_s)) (p_{r_{j+1} -r}(B_r-B_u) -p_{r_{j} -r}(B_r-B_u)) \right)du  \\
&&\quad  = -\int_{r_j} ^{r_{j+1}}
\left( (r_{j+1} -s) (\theta -r+s-u) + (r_{j+1}-r) (r-s) \right) ^{-\frac 12}|_{u=0}^{u=s} d\theta\\
&& \qquad +\int_{r_j} ^{r_{j+1}}
\left( (r_{j} -s) (\theta -r+s-u) + (r_{j}-r) (r-s) \right) ^{-\frac 12}|_{u=0}^{u=s} d\theta\\
&&\quad =-\frac 12 \int_{r_j} ^{r_{j+1}}\int_{r_j} ^{r_{j+1}} \Bigg(
\left((\eta -s) (\theta -r+s) + (\eta-r) (r-s) \right)^{-\frac 32} \theta\\
&&\qquad - \left((\eta -s) (\theta -r) + (\eta-r) (r-s) \right)^{-\frac 32}( \theta-s) \Bigg)d\eta d\theta\\
&&\quad \le  C \int_{r_j} ^{r_{j+1}}\int_{r_j} ^{r_{j+1}}  
\left((r -s) (\theta -r ) + (\eta-r) (r-s) \right)^{-\frac 32} d\eta d\theta\\
&&\quad \le C (r-s) ^{-\frac 32}\left( \int_{r_j} ^{r_{j+1}} (\eta-r) ^{-\frac 34}  d\eta
\right)^2 \\
&&\quad \le C (r-s) ^{-\frac 32}  \left(  (r_{j+1}-r)^{\frac14} -(r_{j }-r)^{\frac14} \right)^2\\
&&\quad \le C(r-r(N)) ^{-\frac 34} (r(N)-s)^{-\frac 34}  (r_{j+1} -r_j)^{ 2-\frac 32 \alpha }  (r_j-r)^{ -\frac 32 (1-\alpha)},
\end{eqnarray*}
for any  $\alpha \in (0,1)$. Choosing $\alpha =\frac 14$ and   integrating in the variables $0<s<r(N)<r<t(N)$,  we obtain
\[
E\left(\int_0^{\frac {kT}N} \left( \int_0^{r(N)} (p_{r_{j+1} -r}(B_r-B_s) -p_{r_{j} -r}(B_r-B_s))ds\right)^2dr  \right)^{\frac 23}
\le C_N  (r_{j+1} -r_j)^{-\frac {13}{12}}.
\]
As a consequence,
\[ 
E(S_{n,N}) \le C_N n^{-\frac 1{12}},
\]
which converges to zero as $n$ tends to infinity.

\medskip \noindent
\textbf{Step  3}

Let us compute the  $\frac 43$ variation of the  process $\gamma^{N,1}$ in the interval $I_{k,N}:=\left[\frac {kT}N, \frac {(k+1)T} N \right]$.
Set $\tau_N =\frac {kT}N=t(N)$.  By  the results of  \cite{perkins}, there exists a two-sided Brownian motion $\{W_x, x\in \mathbb{R}\}$ independent of
$\{ B_r, r\ge r_N, L_{\tau_N}^{B_{\tau_N}} \}$ such that for any $x>y, x,y \in \mathbb{R}$,
\[
L_{\tau_N}^x - L_{\tau_N}^y =2 \int_y^x  \sqrt{L^z_{\tau_N}} dW_z + \int_y^x \alpha(z) dz.
\]
Using the fact that the random variables  $\{ B_r, r\ge r_N, L_{\tau_N}^{B_{\tau_N}} \}$ are independent of $W$ we can write for any $r\ge \tau_N$,
\[
L_{\tau_N}^{B_{r}+y}  - L_{\tau_N}^{B_{r} } =2 \int_{B_{r } } ^{B_{r}+y}  \sqrt{L^z_{\tau_N}} dW_z + \int_{B_{r } } ^{B_{r} +y}\alpha(z) dz.
\]
We decompose the process $\gamma^{N,1}$ as follows:
\[
 \gamma^{N,1}=\gamma^{N,2}+\gamma^{N,3}+\gamma^{N,4},
\]
where
 \[
\gamma^{N,2}=\int_{\tau_N}^t        E^\theta  \left( \int_{B_{r } } ^{B_{r} +\theta\sqrt{t-r}}\alpha(z) dz \right)   dB_r,
\]
\[
\gamma^{N,3}=\int_{\tau_N}^t        E^\theta  \left( \int_{B_{r } } ^{B_{r} +\theta\sqrt{t-r}} \left(\sqrt{L^z_{\tau_N}} -
\sqrt{L^{B_{r }}_{\tau_N}} \right)
dW_z \right)   dB_r,
\]
and
\[
\gamma^{N,4}= \sqrt{L^{B_{\tau_N}}_{r_N}} \int_{\tau_N}^t        E^\theta  \left( W(B_{r} +\theta\sqrt{t-r})-
W(B_r) \right)   dB_r,
\]
where here $\theta$ denotes a random variable with law $N(0,1)$, independent of $B$ and $W$.
We claim that  for any $k$,
\begin{equation}
 \left \langle \gamma^{N,2} \right\rangle_{\frac 43, I_{k,N}}=0, \label{eq1}
 \end{equation}
 and
 \begin{equation}
 \left \langle \gamma^{N,3} \right\rangle_{\frac 43, I_{k,N} }=0, \label{eq2}
 \end{equation}
\textit{Proof of (\ref{eq1})}:
With the same notation as in Step 1, set
\[
S_{n,N}:=S^{I_{k,N}}_{\frac 43,n} (  \gamma^{N,2})= \sum_{j=0}^{n-1}|\Delta _{j} (\gamma^{N,2})|^{\frac 43 },
 \]
 where $ \Delta _{j} (\gamma^{N,2}) = \gamma^{N,2}_{r_{j+1}} -  \gamma^{N,2}_{r_{j}}$.
Then
\begin{eqnarray*}
&&  \sum_{j=0}^{n-1}E|\Delta _{j} (\gamma^{N,2})|^{\frac 43}
  = \sum_{j=0}^{n-1}E\Bigg|\int_{\tau_N}^{r_{j+1}} E^\theta  \left(\int_{B_r}^{B_r+\sqrt{r_{j+1}-r}\theta}\alpha(y)dy\right)dB_r \\
  &&\qquad -\int_{\tau_N}^{r_{j}} E^\theta \left(\int_{B_r}^{B_r+\sqrt{r_{j}-r}\theta}\alpha(y)dy\right)dB_r\Bigg|^\frac43 \\
&& \quad =  \sum_{j=0}^{n-1}E\Bigg|\int_{\tau_N}^{r_{j}} E^\theta \left(\int_{B_r+ \sqrt{r_j-r} \theta}^{B_r+\sqrt{r_{j+1}-r}\theta}\alpha(y)dy\right)dB_r\\
&&\qquad +\int_{r_j}^{r_{j+1}} E^\theta \left(\int_{B_r}^{B_r+\sqrt{r_{j+1}-r}\theta}\alpha(y)dy\right)dB_r\Bigg|^\frac43 \\
&& \quad \le  C  \sum_{j=0}^{n-1}\left\{\left|\int_{\tau_N}^{r_j}E\left( E^\theta \int_{B_r+ \sqrt{r_j-r} \theta}^{B_r+\sqrt{r_{j+1}-r}\theta}\alpha(y)dy\right)^2dr\right|^\frac23\right.\\
&&\qquad +\left.\left|\int_{r_j}^{r_{j+1}}E\left( E^\theta \int_{B_r}^{B_r+\sqrt{r_{j+1}-r}\theta}\alpha(y)dy\right)^2dr\right|^\frac23 \right\}\\
&&\quad =A_n +B_n.
\end{eqnarray*}
 From \cite{perkins}, we have the following expression for the process $\alpha(y)$,
\[
\alpha(y)=I_{\{y\ge\underline B_s\}}\left[2I_{\{y\le0\}}+2I_{\{y\le B_s\}}+I_{\{y\le \overline B_s\}}L(s,y)
\left(\frac{4I_{\{y\ge B_s\}}}{L(s,y)+2y^-}-\frac{L(s,y)+2y^-}{s-A(s,y)}\right)\right]
\]
with
\[
\overline B_s=\sup\{B_u, u\le s\},\underline B_s=\inf\{ B_u:u\le s\}.
\]
Let $\gamma(y)=-I_{\{y\ge\underline B_s\}}I_{\{y\le \overline B_s\}}L(s,y)\frac{L(s,y)+2y^-}{s-A(s,y)}$, and write $\alpha(y)=\beta(y)+\gamma(y)$.  Then $\beta(y)$ is bounded, and
from the result of section 3 (page 277 and 278) in \cite{rw}, we can get that $E\int_{\mathbb R} |\gamma(y)|^pdy<\infty$ for all $ p>1$.
As a consequence,  by Lemma \ref{lemma1} we obtain
\begin{eqnarray*}
&&\lim_{n\to\infty}\sum_{j=0}^{n-1}\left|\int_{\tau_N}^{r_j}E\left( E^\theta \int_{B_r+ \sqrt{r_j-r} \theta}^{B_r+\sqrt{r_{j+1}-r}\theta}\beta(y)dy\right)^2dr\right|^\frac23\\
&& \qquad \le  C \lim_{n\to\infty}\sum_{j=0}^{n-1}\left|\int_{\tau_N}^{r_j}\left(\sqrt{r_{j+1}-r}-\sqrt{r_j-r}\right)^2  dr\right|^\frac23=0.
\end{eqnarray*}
To handle the term containing $\gamma(y)$, we choose $p,q$ such that $\frac 1p+\frac 1q=1$ and $p<\frac 43$. Then, again by Lemma \ref{lemma1}
\begin{eqnarray*}
&&\lim_{n\to\infty}\sum_{j=0}^{n-1}\left|\int_{\tau_N}^{r_j}E\left( E^\theta \int_{B_r+ \sqrt{r_j-r} \theta}^{B_r+\sqrt{r_{j+1}-r}\theta}\gamma(y)dy\right)^2dr\right|^\frac23\\
&&\qquad \le  \lim_{n\to\infty}\sum_{j=0}^{n-1}\left|\int_{\tau_N}^{r_j}  E(|\theta|^\frac 2p) (\sqrt{r_{j+1}-r}-\sqrt{r_j-r})^\frac 2p E\left[\int_\mathbb R |\gamma(y)|^qdy\right]^\frac 2q  dr\right|^\frac23\\
&&\qquad \le  C\lim_{n\to\infty} \sum_{j=0}^{n-1}\left|\int_{\tau_N}^{r_j} (\sqrt{r_{j+1}-r}-\sqrt{r_j-r})^\frac 2p   dr\right|^\frac23=0.
\end{eqnarray*}
Hence we have $A_n$ goes to zero as $n$  goes to infinity. The convergence to zero  of  $B_n$  as $n$ tends to infinity follows from
\begin{eqnarray*}
&&\lim_{n\to\infty}\sum_{j=0}^{n-1}\left| \int_{r_j}^{r_{j+1}} E\left( E^\theta \int_{B_r}^{B_r+\sqrt{r_{j+1}-r}\theta}\beta(y)dy\right)^2dr\right|^\frac23\\
&&\qquad \le  C \lim_{n\to\infty}\sum_{j=0}^{n-1}\left|\int_{r_j}^{r_{j+1}} \left(r_{j+1}-r\right)  dr\right|^\frac23=C \lim_{n\to\infty}\sum_{j=0}^{n-1} \left(\frac 1n\right)^\frac 43=0,
\end{eqnarray*}
and, choosing  $p,q$ such that $\frac 1p+\frac 1q=1$ and $p<2$,
\begin{eqnarray*}
&&\lim_{n\to\infty}\sum_{j=0}^{n-1}\left|\int_{r_j}^{r_{j+1}} E\left( E^\theta \int_{B_r}^{B_r+\sqrt{r_{j+1}-r}\theta}\gamma(y)dy\right)^2dr\right|^\frac23\\
&& \quad \le  \lim_{n\to\infty}\sum_{j=0}^{n-1}\left|\int_{r_j}^{r_{j+1}}   E(|\theta|^\frac 2p) (\sqrt{r_{j+1}-r})^\frac 2p E\left[\int_\mathbb R |\gamma(y)|^qdy\right]^\frac 2q  dr\right|^\frac23\\
&& \quad \le  C\lim_{n\to\infty} \sum_{j=0}^{n-1}\left|\int_{r_j}^{r_{j+1}}  (r_{j+1}-r)^\frac 1p   dr\right|^\frac23= C\lim_{n\to\infty} \sum_{j=0}^{n-1} \left(\frac 1n\right)^{(\frac 1p+1)\frac 23}= 0.
\end{eqnarray*}

\noindent
\textit{Proof of (\ref{eq2})}:
With the same notation as in Step 1, set
\[
S_{n,N}:=S^{I_{k,N}}_{\frac 43,n} (  \gamma^{N,3})= \sum_{j=0}^{n-1}|\Delta _{j} (\gamma^{N,3})|^{\frac 43 },
 \]
 where $ \Delta _{j} (\gamma^{N,3}) = \gamma^{N,3}_{r_{j+1}} -  \gamma^{N,3}_{r_{j}}$. As in the proof of (\ref{eq1}), applying the Burkholder  inequality we obtain
 \[
  \sum_{j=0}^{n-1}E|\Delta _{j} (\gamma^{N,3})|^{\frac 43}
\le C( C_n+D_n),
\]
 where
 \[
 C_n=E\sum_{j=1}^{n-1} \left(\int^{r_j}_ {r_N}     E^\theta  \left(
\int_{B_r+\theta \sqrt{r_{j }-r}}^{B_r +\theta \sqrt{r_{j+1}-r}}   \left(\sqrt{ L_{\tau_N}^{ z }  }-     \sqrt{ L_{\tau_N}^{ B_{r}}  }   \right)  dW_z     \right)^2 dr  \right) ^{2/3}
 \]
and
\[
D_n=E\sum_{j=1}^{n-1} \left(\int^{r_{j+1}}_ {r_j}     E^\theta  \left(
\int_{B_r}^{B_r +\theta \sqrt{r_{j+1}-r}}   \left(\sqrt{ L_{\tau_N}^{ z }  }-     \sqrt{ L_{\tau_N}^{ B_{r}}  }   \right)  dW_z     \right)^2 dr  \right) ^{2/3}.
 \]
 By the H\^older continuity in space variable  of the local time, there exists a random variable $G$ with moments of all orders such that
 \[
 |L_{\tau_N}^{ z }- L_{\tau_N}^{ B_{r}}| \le G |z-B_r|^{\frac 12-\epsilon},
 \]
 for all $z\in\mathbb{R}$ and $r\in [0,T]$.  Therefore, the term $D_n$ can be estimated as follows
 \[
D_n \le C \sum_{j=1}^{n-1} \left(\int^{r_{j+1}}_ {\tau_{r_j}}    E E^\theta  \left(  G
\int_{B_r}^{B_r +\theta \sqrt{r_{j+1}-r}}   |z-B_r|^{\frac 12-\epsilon}   dz     \right)  dr  \right) ^{2/3} \le C n^{-\frac 16+\frac\vare 3}.
 \]
 The estimation of the term $C_n$ is more delicate. First we write
 \[
    E^\theta  \left(
\int_{B_r+\theta \sqrt{r_{j }-r}}^{B_r +\theta \sqrt{r_{j+1}-r}}   \left(\sqrt{ L_{\tau_N}^{ z }  }-     \sqrt{ L_{\tau_N}^{ B_{r}}  }   \right)  dW_z    \right)
=\int_{\mathbb{R}} \Phi(z) dW_z,
\]
where
\[
\Phi(z)= \left(\sqrt{ L_{\tau_N}^{ z }  }-     \sqrt{ L_{\tau_N}^{ B_{r}}  }   \right)
\int_{\mathbb{R}} \left( p_{r_{j+1}-r}(y)- p_{r_{j}-r}(y) \right)  \mathbf{1}_{[B_r, B_r+y]}(z)dy.
\]
As a consequence,
\begin{eqnarray*}
&&E\left(\int_{\mathbb{R}} \Phi(z) dW_z\right)^2\le
E\Bigg( G^2  \int_{\mathbb{R}^2}
\left( p_{r_{j+1}-r}(y)- p_{r_{j}-r}(y) \right)
\left( p_{r_{j+1}-r}(y')- p_{r_{j}-r}(y') \right) \\
&&\qquad \times
 \int_{[B_r, B_r+y]\cap [B_r,B_r+y']} |z-B_r|^{\frac 12-\epsilon}  dzdydy'\Bigg) \\
 && \quad \le C \left( \int_{\mathbb{R}} \left( p_{r_{j+1}-r}(y)- p_{r_{j}-r}(y) \right)
 |y|^{\frac 34 -\frac \epsilon 2} dy\right)^2 \\
 &&\quad \le C \left( (r_{j+1} -r) ^{\frac 38-\frac \epsilon 4}-
 (r_{j} -r) ^{\frac 38-\frac \epsilon 4} \right)^2   = C \left(  \int_{r_j}^{r_{j+1}} (\th-r) ^{-\frac 58-\frac \epsilon 4}  d\th \right)^2 \\
 && \quad \le C \left(  \int_{r_j}^{r_{j+1} } (\th-r_j) ^{-\frac 14+\frac \epsilon 2}   d\th \
  ( r_j-r )^ {-\frac 38-\frac {3\epsilon  }  4 } \right)^2  \\
 && \quad \le C (r_{j+1} -r_j) ^{\frac 32+ \epsilon}
 (r_j-r)^{-\frac 34 -\frac {3\epsilon}2},
 \end{eqnarray*}
 and we obtain
 \[
 C_n \le C n^{-\frac 23  \epsilon}.
 \]
This proves \eref{eq2}.

 \medskip\noindent
\textbf{Step  4}

Let us compute the $\frac 43$-variation of the process $\gamma^{N,4}$. By
Theorem  \ref{th2}, the $\frac 43$ variation in $L^1$ of the process
\[
Z_t=\int_0^tE^\theta (W_{B_r+ \theta\sqrt{t-r}} - W_{B_r})dB_r,
\]
in an interval $[a,b]$ is $K(b-a)$. In fact, this process has the same  distribution as
\[
X_t=\int_0^tE^\theta (W_{ \theta\sqrt{t-r}}  )dB_r.
\]
This follows from the fact that the processes
\[
\{  (B_t, W_{B_r+ y} - W_{B_r}), t\ge 0,  r\ge 0, y\in \mathbb{R} \}
\]
and
\[
\{  (B_t, W_y), t\ge 0,  r\ge 0, y\in \mathbb{R} \}
\]
have the same law, as it can be easily seen by computing the characteristic function of the finite dimensional distributions of both processes. Therefore,
\[
\langle \gamma^{N,4} \rangle_{\frac 43 ,[0,T]} = K\sum_{k=0}^{N-1} (L_{kT/N}^{B_{kT/N}})^{\frac 23} \frac TN.
\]
By Step 2 and Step 3, we have that $\langle \gamma^N\rangle _{\frac 43, [0,T]}
=\langle \gamma^{N,4}\rangle _{\frac 43, [0,T]}$.   Then the proof of Theorem \ref{th1} follows immediately from Step 1 and the fact that
\[
\lim_{N\rightarrow \infty}  \langle \gamma^{N} \rangle_{\frac 43 ,[0,T]}
= K \int_0^T \left( L_r^{B_r} \right) ^{\frac 23} dr.
\]

\section{Appendix}
\begin{lemma} \label{A1}
Let $0\le a<b<c<d$, and set $x=b-a$, $y=c-b$ and $z=d-c$. Then,
\begin{eqnarray*}
&&E \left[ E\left( \mathbf{1}_{\{ B_{c} <B_a < B_{d} \}}   - \mathbf{1}_{\{ B_{c} >B_a> B_d \}}  | \mathcal{F}_b  \right)  \right]^2  \\
&&\qquad  \le C x^{-\frac 12}  \left( 2\sqrt{2y+z}
-\sqrt{2y} -\sqrt{ 2y +2z} \right).
\end{eqnarray*}
\end{lemma}

\begin{proof}
Set
\[
B_a-B_b= \sqrt{x} X, \quad B_{c}-B_b=\sqrt{y}  Y, \quad
B_{d}- B_{c} = \sqrt{z} Z,
\]
where $X$, $Y$ and $Z$ are independent $N(0,1)$ random variables.
With this notation we can write
\begin{eqnarray*}
&&E \left[ E\left( \mathbf{1}_{\{ B_{c} <B_a < B_{d} \}}   - \mathbf{1}_{\{ B_{c} >B_a> B_d \}}  | \mathcal{F}_b  \right)  \right]^2 \\
 &=& E\Big[P(\sqrt{y} Y <\sqrt{x} X <   \sqrt{z}   Z+ \sqrt{y}  Y|X)
  - P(\sqrt{y}   Y >\sqrt{x} X>   \sqrt{z}   Z+ \sqrt{y}   Y|X) \Big]^2 \\
&=& \int_{\mathbb{R}}  \left(\int_{\mathbb{R} } \phi(\eta) d\eta   \int^{\sqrt{\frac xy} \theta}  _{ \sqrt{\frac xy}\theta- \sqrt{\frac zy}\eta} \phi(\xi) d\xi
\right)^2 d\theta,
\end{eqnarray*}
where $\phi(x) $ is the density of the law $N(0,1)$.  Set
\[
g(x,y,z,\theta)=\int_{\mathbb{R} } \phi(\eta) d\eta   \int^{\sqrt{\frac xy} \theta}  _{ \sqrt{\frac xy}\theta- \sqrt{\frac zy}\eta} \phi(\xi) d\xi.
\]
Then,
\begin{eqnarray*}
g(x,y,z,\theta)
 &=& \frac 1{\sqrt{y}}\int_0^{\sqrt{z}}\int_\mathbb{R} \phi(\eta)     \phi(\sqrt{\frac xy} \theta- \frac {w}{\sqrt{y}} \eta ) \eta d\eta  dw \\
 &=& \frac 1 {2\pi}  \frac 1{\sqrt{y}}   \int_0^{\sqrt{z}}\int_\mathbb{R}   \exp\left(- \frac 12        (\eta^2 +  (\sqrt{\frac xy}\theta- \frac {z}{\sqrt{y}} \eta )^2 ) \right)  \eta d\eta dw\\
 &=& \frac 1 {2\pi}  \int_0^{\sqrt{z}}  \frac {w \sqrt{x} \theta}{(y+w^2)^{\frac 32}}  \exp \left( -\frac {x\theta^2}{2(y+w^2)} \right)   dw\\
 &=& \frac 1 {4\pi}  \int_0^{z}  \frac { \sqrt{x} \theta}{(y+\xi)^{\frac 32}}  \exp \left( -\frac {x\theta^2}{2(y+\xi)} \right)   d\xi.
\end{eqnarray*}
 Finally, integrating with respect to $\theta$ yields
 \begin{align*}
&\int_\mathbb R g(x,y,z,\theta)^2\phi(\theta)d\theta\\
=&Cx\int_\mathbb R \int_0^{z}\int_0^{z} \frac{ \theta^2}{(y+\xi_1)^\frac 32(y+\xi_2)^\frac 32}
\exp\left(-\frac12\left(\frac{x\theta^2}{y+\xi_1}+\frac{x\theta^2}{y+\xi_2}\right)\right)d\xi_1d\xi_2 \phi(\theta)d\theta\\
=&Cx \int_0^{z}\int_0^{z} \frac{1}{(y+\xi_1)^\frac 32(y+\xi_2)^\frac 32}\int_\mathbb R \theta^2
\exp\left(-\frac{\theta^2}{2} \left(\frac{x}{y+\xi_1}+\frac{x}{y+\xi_2}+1\right)\right)d\theta d\xi_1 d\xi_2 \\
=&Cx \int_0^{z}\int_0^{z}  \frac{1}{(y+\xi_1)^\frac 32(y+\xi_2)^\frac 32}\left(\frac{x}{y+\xi_1}+\frac{x}{y+\xi_2}+1\right) ^{-\frac 32} d\xi_1d\xi_2 \\
=&Cx \int_0^{z}\int_0^{z}   \left[x(2y+\xi_1 +\xi_2)+ (y+\xi_
1)(y+\xi_2)\right]^{-\frac 32}  d\xi_1 d\xi_2 \\
\le&Cx^{-\frac12}   \int_0^{z}\int_0^{z}    (2y+\xi_1 +\xi_2) ^{-\frac 32}  d\xi_1 d\xi_2 \\
=&C   x^{-\frac12} \left[    2\sqrt{ 2y+z} - \sqrt{2y} -\sqrt{2y+2z} \right],
\end{align*}
which completes the proof of the lemma.
\end{proof}

\begin{lemma} \label{A2}
Let $\alpha, \beta >0$ and let $X$, $Y$ be independent random variables with laws
$N(0,\sigma_1^2)$ and  $N(0,\sigma_2^2)$, respectively. Then
\[
E\left[ p_\alpha(X) p_\beta(X+Y)\right]=\left( (\alpha + \sigma_1^2)
(\beta+ \sigma_2^2)+ \alpha \sigma_1^2  \right) ^{-\frac 12}.
\]
\end{lemma}

\begin{lemma}\label{lemma1}
 Suppose $a<b$ and $n\in \mathbb N.$ Let $r_j=a+\frac j n(b-a), j=0,1,\dots,n.$ Then, for any $\beta >\frac 32$, we have
\[
\lim_{n\to\infty} \sum_{j=1}^n \left|\int_a^{r_j} (\sqrt{r_{j+1}-r}-\sqrt{r_{j}-r})^\beta dr\right|^\frac 23=0.
\]
\end{lemma}

\begin{proof}
It suffices to use the estimate
\[
\sqrt{r_{j+1}-r}-\sqrt{r_{j}-r} \le C (r_{j+1}-r_j)^{\frac 12+ \frac 3{4\beta}}
(r_j-r)^{-\frac 3{4\beta}}.
\]
\end{proof}

\medskip
\noindent
\textbf{Acknowledgements} We would like  to thank Jay Rosen for having suggested this problem to us.

\medskip
\parindent=0pt
Yaozhong Hu and  David Nualart\\
Department of Mathematics \\
University of Kansas \\
Lawrence, Kansas, 66045 \\
and\\
Jian Song \\
Department of Mathematics\\
Rutgers University\\
Hill Center - Busch Campus\\
110 Frelinghuysen Road\\
Piscataway, NJ 08854-8019

\end{document}